\documentclass[letterpaper, 10 pt, conference]{ieeeconf} 

\IEEEoverridecommandlockouts                              

\usepackage{graphics} 
\usepackage{epsfig} 
\usepackage{amsmath} 
\usepackage{amssymb}  
\usepackage{tikz}
\usepackage{mathtools}
\usetikzlibrary{arrows.meta,positioning,calc}

\title{\LARGE \bf
Nested Extremum Seeking Converges to Stackelberg Equilibrium
}

\author{Brad Ratto$^{1,2}$, Alan Williams$^2$, Miroslav Krstić$^1$, Tamer Ba\c{s}ar$^3$, Alexander Scheinker$^2$
\thanks{This work was funded by the U.S. Department of Energy (DOE) Los Alamos National Laboratory LDRD Program Directed Research (DR) Project No. 20220074DR.}
\thanks{$^{1}$Brad Ratto and Miroslav Krstić are with the Department of Mechanical and Aerospace Engineering, University of California San Diego, La Jolla, CA 92093 USA (e-mail: ratto.brad@lanl.gov; krstic@ucsd.edu).}
\thanks{$^{2}$Brad Ratto, Alan Williams, and Alexander Scheinker are with Los Alamos National Laboratory, Los Alamos, NM 87545, USA (e-mail: awilliams@lanl.gov, ascheink@lanl.gov).}
\thanks{$^{3}$Tamer Ba\c{s}ar is with the Department of Electrical and Computer Engineering, University of Illinois Urbana--Champaign, Urbana, IL 61801, USA (e-mail: basar1@illinois.edu).}}

\begin{document}

\newtheorem{remark}{Remark}
\newtheorem{definition}{Definition}
\newtheorem{assumption}{Assumption}
\newtheorem{theorem}{Theorem}
\newtheorem{lemma}{Lemma}
\newtheorem{proposition}{Proposition}
\newtheorem{corollary}{Corollary}

\maketitle
\thispagestyle{empty}
\pagestyle{empty}

\begin{abstract}
The nested Extremum Seeking (nES) algorithm is a model-free optimization method that has been shown to converge to a neighborhood of a Nash equilibrium. In this work, we demonstrate that the same nES dynamics can instead be made to converge to a neighborhood of a Stackelberg (leader--follower) equilibrium by imposing a different scaling law on the algorithm's design parameters. For the two--level nested case, using Lie--bracket averaging and singular perturbation arguments, we provide a rigorous stability proof showing semi-global practical asymptotic convergence to a Stackelberg equilibrium under appropriate time-scale separation. The results reveal that equilibrium selection, Nash versus Stackelberg, depends not on modifying the closed-loop dynamics, but on the hierarchical scaling of design parameters and the induced time-scale structure. We demonstrate this effect using a simple quadratic example and the canonical Fish War game. The Stackelberg variant of nES provides a model-free framework for hierarchical optimization in multi-time-scale systems, with potential applications in power grids, networked dynamical systems, and tuning of particle accelerators.
\end{abstract}

\section{Introduction}

Extremum Seeking (ES) is a real-time, model-free optimization method that adjusts system parameters to minimize or maximize an unknown but measurable performance metric. ES has been applied to optimizing energy systems, regulating biological processes, operating particle accelerators, and stabilizing nonlinear systems \cite{ghaffari2014power, manzie2009extremum, scheinker2020online, scheinker2012minimum,scheinker2013extremum,scheinker2014extremum,scheinker2014hardware}. ES has also been incorporated within the architectures of generative diffusion models for time-varying systems \cite{scheinker2024cdvae}, and with guaranteed safety constraints \cite{williams2026local, williams2024semiglobal}.

\begin{figure*}[t!]
\centering
\includegraphics[width=\linewidth]{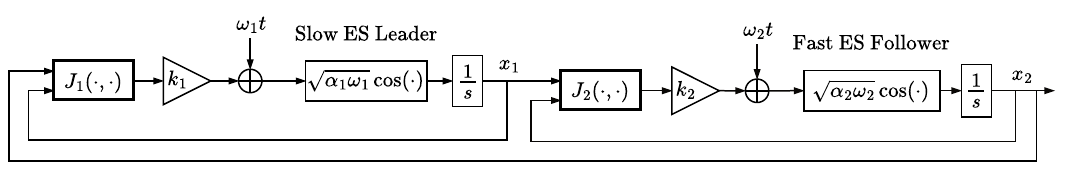}
\caption{Block diagram of the nES architecture for $n=2$ nestings. The Slow ES Leader tunes $x_1$, and the Fast ES Follower tunes $x_2$.}\label{fig:nes-diag}
\end{figure*}

Because ES handles uncertainty and complex dynamics, it has been used to coordinate interacting agents, including Nash equilibrium (NE) seeking in noncooperative games for mobile sensor networks \cite{stankovic2011distributed} and more general nonlinear dynamics-driven settings \cite{frihauf2011nash}. Recent surveys document model-free Nash seeking via ES and extensions to delays and PDE dynamics \cite{frihauf2011nash}. Complementary to NE, another non-cooperative equilibrium solution is the Stackelberg equilibrium (SE) \cite{basar1999dynamic}. For coupled systems, it is often advantageous or necessary to tune parameters according to a hierarchy of objectives. A game admits a SE when a hierarchical structure is imposed in which the leader (player 1) chooses their action $x_1$ first; the follower (player 2), having observed $x_1$, chooses their best response action $x_2$ by solving the following optimization problem: 
\begin{equation}
\min_{x_2\in\mathcal{X}_2} J_{\rm{F}}(x_1,x_2),
\label{eq:follower-min}
\end{equation}
where $J_{\rm{F}}: {\cal X}_1\times {\cal X}_2 \rightarrow \mathbb{R}$ is the cost function of the follower, and ${\cal X}_i \subset \mathbb{R}^{n_i}$ are the action sets of the two players. Let $T_{\rm F}(x_1)$ denote the follower's best response function given by \eqref{eq:follower-min}, which is assumed to be unique for each $x_1\in{\cal X}_1$. We denote a SE by $(x_1 ^{\textrm S},x_2 ^{\textrm S})$, which is obtained by solving
\begin{align}
x_1 ^{\textrm S} \in \arg\min_{x_1\in\mathcal{X}_1} J_{\rm{L}}\big(x_1,T_{\rm F}(x_1)\big), \quad x_2 ^{\textrm S} = T_{\rm F}(x_1 ^{\textrm S}). \label{eq:x^S condition} 
\end{align}
That is, the leader finds its optimal action with respect to $J_{\rm{L}}: {\cal X}_1\times {\cal X}_2 \rightarrow \mathbb{R}$ given the best response of the follower \cite{basar1999dynamic,simaan1973stackelberg}.

Such leader–follower games appear in many applications: smart grids and networked systems \cite{mojica2022stackelberg}, hierarchical control/economic regulation settings modeled via dynamic games \cite{basar1999dynamic}, demand-response and pricing in smart grids \cite{Maharjan2013TSG,Alshehri2015CDC}, power control and channel allocation in cognitive radio networks \cite{BloemAlpcanBasar2007GameComm}, and mixed leadership in economics \cite{BasarHaurieRicci1985JEDC}. From a computational point of view, while NE seeking is well studied, comparatively less work has addressed Stackelberg games.

One early approach used a genetic algorithm for off-line computation of SE \cite{vallee1999off}; this was later complemented by an online approach using neural networks to estimate the follower best-response map \cite{alemdar2003line}. Other examples include two–time–scale gradient descent–ascent, which proves local stability/convergence under time-scale separation, yielding a fast follower-slow leader scheme \cite{fiez2019convergence,fiez2021timescale}; Stackelberg actor–critic in reinforcement learning, which runs follower/critic updates fast and leader policy updates slow \cite{zheng2022stackelberg}; and population/game-theoretic control formulations that analyze leader coordination with fast follower dynamics \cite{mojica2022stackelberg}.

In this paper, we adopt the same separation-of-time-scales concept but implement a \emph{model-free} approach via ES. We leverage the nested Extremum Seeking (nES) architecture previously shown to be a model-free method that converges practically to a NE \cite{Ratto2026ACC}. The nES framework employs nested feedback loops, each coordinating the optimization of multiple parameters with respect to different objective functions. Our principal contribution is a permutation of Lie--bracket averaging and singular perturbation tools that leads to a different convergence result than in the Nash nES seeking case \cite{Ratto2026ACC}. A similar singularly perturbed Lie--bracket tool is presented in \cite{durr2017extremum}; however, our analysis requires an additional Lie--bracket averaging step. The stability proof presented is for the $n=2$ nested case, which captures the core ideas and paves the way for generalization. 

To prove convergence and stability, we perform a three-stage analysis. First, we apply Lie--bracket averaging to the fast state. Then, we use singular perturbation to define a boundary--layer model (BLM), and in turn the reduced order model (ROM) and quasi--steady state (QSS). Finally, we apply Lie--bracket averaging to the ROM, which allows us to define the averaged ROM. By analyzing the stability of the BLM and the averaged ROM, we use a stability result regarding averaged systems to claim convergence of the original system to a neighborhood of the SE under standard ES assumptions. Re-framed in leader–follower terms, the Fast ES Follower tracks the best-response map $x_2=T_{\rm F}(x_1)$ by minimizing $J_{\rm F}(x_1,x_2)$ at the current $x_1$, while the Slow ES Leader adjusts $x_1$ to reduce $J_{\rm L}(x_1,T_{\rm F}(x_1))$ along the manifold $T_{\rm F}(x_1)$.

\section{Preliminaries}
We use this section to introduce concepts and develop the necessary tools used in the analysis of the nES algorithm.

\subsection{Strong Monotonicity and Stability Lemma}

\begin{definition}[Strong monotonicity]
\label{def:strong-monotone}
Let $q(x):\mathbb{R}\to\mathbb R$ be continuously differentiable. We say that $q(x)$ is \emph{$\mu$--strongly monotone on $\mathbb{R}$} if there exists $\mu>0$ such that
\begin{equation}
q'(x)\ge \mu,\qquad \forall x\in\mathbb{R}.
\end{equation}
\end{definition}

\begin{lemma}
\label{lem:scalar-exp-stable}
Consider $q(x):\mathbb{R}\to\mathbb R$ to be continuously differentiable and $\mu$--strongly monotone on $\mathbb{R}$ for some $\mu>0$. Let $q(x^*)=0$ for some $x^*\in\mathbb{R}$. Then, $x^*$ is the unique exponentially stable equilibrium of the system
\begin{equation}
\dot x=-c\,q(x),
\end{equation}
where $c>0$. In particular, 
\begin{equation}
|x(t)-x^*|\le |x(t_0)-x^*|\,e^{-c\mu (t-t_0)},\qquad \forall t\ge t_0.
\end{equation}
\end{lemma}

\begin{proof}
Since $q'(x)\ge\mu>0$ for all $x\in\mathbb{R}$, $q$ is strictly increasing on $\mathbb{R}$; hence $q(x)=0$ has at most one root, and thus $x^*$ is unique. Consider $V(x)=\frac12|x-x^*|^2$. Along trajectories,
\begin{equation}
\dot V=-c(x-x^*)\big(q(x)-q(x^*)\big).
\end{equation}
By the mean value theorem, $q(x)-q(x^*)=q'(\tilde x)(x-x^*)$ for some $\tilde x$ between $x$ and $x^*$; hence
\begin{equation}
(x-x^*)\big(q(x)-q(x^*)\big)=q'(\tilde x)|x-x^*|^2\ge \mu|x-x^*|^2.
\end{equation}
Therefore $\dot V\le -c\mu|x-x^*|^2=-2c\mu V$, and by \cite[Lemma 3.4]{khalil_nonlinear_2002} (Comparison Lemma)
    \begin{equation}
        V(x(t))\le V(x(t_0))e^{-2c\mu(t-t_0)}, \quad \forall t \ge t_0.
    \end{equation}
We then obtain $|x(t)-x^*| \leq | x(t_0)-x^* | e^{-c\mu (t-t_0)}.$
\end{proof}

\subsection{CTP, GUAS, and SPUAS Definitions}
We now define the converging trajectories property (CTP), a standard attribute used to quantify the convergence of trajectories between the original and averaged forms of a system, as discussed in \cite{moreau2002practical}. 
\begin{definition}[CTP]
\label{def:CTP}
The dynamics $\dot{\bar x}=f(t,\bar x)$ and $\dot x=f^\upsilon(t,x)$ with corresponding solutions $\psi$ and $\psi^\upsilon$ satisfy the CTP if for any compact set $\mathcal{K}\subseteq\mathbb{R}^n$, $T>0$, $\delta>0$, there exists $\upsilon^*>0$ such that for all $t_0\in\mathbb{R}$, for all $x_0\in \mathcal{K}$, and for all $0<\upsilon<\upsilon^*$,
\begin{equation}
\|\psi^\upsilon(t,t_0,x_0)-\psi(t,t_0,x_0)\|<\delta, \quad \forall t\in[t_0,t_0+T].
\end{equation}
\end{definition}
Next, we define the notions of global uniformly asymptotic stability (GUAS) and semi-global practical uniform asymptotic stability (SPUAS), as presented in \cite{moreau2002practical,scheinker2016bounded, teel1998global}. 
\begin{definition}[GUAS]
\label{def:GUAS}
Let $x^*$ be an equilibrium of the system $\dot{\bar x}=f(t,\bar{x})$. The system is GUAS if there exists $\beta\in\mathcal{KL}$ such that for every initial condition $x_0\in \mathbb{R}^n$,
\begin{equation}
\|\bar x(t)-x^*\|\le \beta(\|x_0-x^*\|,t-t_0), \qquad \forall t\ge t_0,
\end{equation}
for any $t_0\in\mathbb R$.
\end{definition}
\begin{definition}[SPUAS]
\label{def:SPUAS}
The system $\dot x=f^\upsilon(t,x)$ is $\upsilon$-SPUAS with residual $\sigma(\upsilon)$ if there exist $x^*\in\mathbb{R}^n$, $\beta\in\mathcal{KL}$,
$\sigma\in\mathcal{K}_\infty$, and $\upsilon^*>0$, such that for every $0<\upsilon<\upsilon^*$,
\begin{equation}
\|x(t)-x^*\|\le \beta(\|x(t_0)-x^*\|,t-t_0)+\sigma(\upsilon),
\end{equation}
$\forall t\ge t_0$, and $x(t_0) \in \mathbb{R}^n$.
\end{definition}

\subsection{Lie--Bracket Averaging}
We now introduce a set of theorems that establish properties between input-affine systems and their Lie--bracket--averaged representations. Consider systems of the form
\begin{equation} \label{eq:gen_form_for_LBA}
    \dot x = b_0(t,x) + \sum_{i=1}^m b_i(t,x) \sqrt{\omega} u_i(t, \omega t).
\end{equation}
We define a corresponding Lie--bracket system 
\begin{equation} \label{eq:gen_LBA_system}
    \dot z = b_0(t,z) + \sum_{1\le i<j\le m} [b_i,b_j](t,z)  \nu_{ij},
\end{equation}
where
\begin{equation}
    \nu_{ji}   =   \frac{1}{2\pi}\int_{0}^{2\pi} u_j(t,\theta) \bigg(\int_{0}^{\theta} u_i(t,\tau) d\tau\bigg)  d\theta,
\end{equation}
and the Lie--bracket of $C^1$ vector fields $f,g$ is
\begin{equation}
    [f,g](t,x)   :=   \frac{\partial g(t,x)}{\partial x} f(t,x) - \frac{\partial f(t,x)}{\partial x} g(t,x).
\end{equation}
We now restate the result from \cite[Theorem 1]{durr2013lie}, which asserts that the solutions of \eqref{eq:gen_form_for_LBA} and \eqref{eq:gen_LBA_system} satisfy the CTP.

\begin{theorem}
    \label{thm:Dur-Avg}
    Consider an arbitrary compact set $\mathcal{K}\subseteq\mathbb{R}^n$, and let the following assumptions be satisfied:
    \begin{itemize}
        \item  the vector fields $b_i(t,x)\in C^2:\mathbb{R}\times\mathbb{R}^n\rightarrow\mathbb{R}^n$ for all $x\in \mathcal{K}$ and $t\in\mathbb{R}$,
        \item the inputs $u_i(t,\theta):\mathbb{R}\times\mathbb{R}\rightarrow\mathbb{R}$ are Lipschitz for all $t\in\mathbb{R}$ and bounded for all $t,\theta\in\mathbb{R}$,
        \item the inputs $u_i(t,\cdot)$ are $2\pi$-periodic in $\theta$, and have zero mean over a period for all $t,\theta\in\mathbb{R}$. 
    \end{itemize}
    Suppose $\mathcal{K}\subseteq \mathcal{B}$, where $\mathcal{B}$ is a set of initial conditions for \eqref{eq:gen_LBA_system} under which \eqref{eq:gen_LBA_system} has unique, uniformly bounded solutions, i.e. there exists $\mathrm A\in(0,\infty)$ such that for all $t_0\in\mathbb{R}$ we have
    \begin{equation}
    z(t_0)\in\mathcal{B}\implies z(t)\in\mathcal{U}_\mathrm A^0, \quad t\ge t_0, 
    \end{equation} 
    where $\mathcal{U}_\mathrm A^0:=\{z(t)\in\mathbb R^n:\inf_{z(t_0)\in\mathcal{B}}|z(t)-z(t_0)|<\mathrm A\}$. Then, for every $\delta\in(0,\infty)$ and $T\in(0,\infty)$, there exists $\omega^*>0$ such that for each $\omega>\omega^*$, for any $t_0\in\mathbb{R}$, and every $x_0\in\mathcal{K}$, there exist unique solutions $x$ and $z$ of \eqref{eq:gen_form_for_LBA} and \eqref{eq:gen_LBA_system} through $x(t_0)=z(t_0)=x_0$, which satisfy  
    \begin{equation}  
    |x(t)-z(t)|<\delta, \quad t\in[t_0,t_0+T].
    \end{equation}
\end{theorem}

\subsection{GUAS $+$ CTP $\Rightarrow$ SPUAS}
The following result is restated from \cite[Theorem 1]{scheinker2016bounded}, which states that \eqref{eq:gen_form_for_LBA} is SPUAS provided that the CTP holds and that \eqref{eq:gen_LBA_system} is GUAS.
\begin{theorem}
\label{thm:MA-bridge}
Suppose the systems $\dot x = f^\upsilon(t,x)$ and $\dot{\bar x}=f(t,\bar x)$ satisfy the CTP. If $x^*$ is GUAS for $\dot{\bar x}=f(t,\bar x)$, then $x^*$ is $\upsilon$-SPUAS for $\dot x=f^\upsilon(t,x)$ for all $x(t_0)\in\mathbb{R}^n$, and $t_0\in\mathbb{R}$.
\end{theorem}

\section{nES Dynamics and Assumptions}

We present the $n=2$ nested case of the nES algorithm, for which a block diagram is shown in Fig.~\ref{fig:nes-diag}, and the dynamics are given by 
\begin{align}
    \dot{x}_1 &= \sqrt{\alpha_1\omega_1}\cos\left(\omega_1 t + k_1 J_{\text{L}}(x_1, x_2)\right), \label{eq:x1}\\
    \dot{x}_2 &= \sqrt{\alpha_2\omega_2}\cos\left(\omega_2 t + k_2 J_{\text{F}}(x_1, x_2)\right), \label{eq:x2}
\end{align}
with initial condition $\left(x_1(t_0),x_2(t_0)\right)=\left(\xi_0,\eta_0\right)$, and design parameters $\alpha_i,k_i>0$, and $\omega_i>0$, for $i\in\{1,2\}$; where $\omega_i=\omega\hat{\omega}_i$ such that $\hat{\omega}_i\neq\hat{\omega}_j$ for $i\neq j$. System \eqref{eq:x1}--\eqref{eq:x2} is the bounded form of ES that was created in \cite{scheinker2013model} and studied for a wide range of systems in \cite{scheinker2014extremum}. We will occasionally refer to $\varepsilon >0$, which is useful in our singular perturbation analysis and is defined as $\varepsilon \coloneqq {1}/{(\alpha_2 k_2)}$.

\begin{figure}[t]
\centering
\begin{tikzpicture}[
    node/.style={draw, rounded corners, align=left, font=\footnotesize, inner sep=6pt, text width=.9\columnwidth},
    lab/.style={font=\scriptsize, align=left},
    >=Latex
]
\node[node] (S0) {%
\textbf{Original system}
\begin{align*}
\dot x_1 &= \sqrt{\alpha_1\omega_1} \cos \big(\omega_1 t+k_1 J_{\text{L}}(x_1,x_2)\big)\\
\dot x_2 &= \sqrt{\alpha_2\omega_2} \cos \big(\omega_2 t+k_2 J_{\text{F}}(x_1,x_2)\big)
\end{align*}
};
\node[node, below=10mm of S0] (S1) {%
\textbf{Partially--Averaged System}
\begin{align*}
\dot{\hat x}_1 &= \sqrt{\alpha_1\omega_1} \cos \big(\omega_1 t+k_1 J_{\text{L}}(\hat x_1, \hat x_2)\big)\\
\dot{\hat x}_2 &= - \frac{\alpha_2 k_2}{2} \partial_{x_2} J_{\text{F}}(\hat x_1, \hat x_2)
\end{align*}
};

\node[node, below=10mm of S1] (S2) {%
\textbf{ROM and QSS} 
\begin{align*}
\dot{\hat x}_1^{\rm{r}} &= \sqrt{\alpha_1\omega_1} \cos \left(\omega_1 t+k_1 J_{\text{L}}\big(\hat x_1^{\rm{r}},h(\hat x_1^{\rm{r}})\big)\right)\\
\hat x_2^{\rm{q}}&=h(\hat x_1^{\rm{r}})+\breve y \left(\frac{t-t_0}{\varepsilon}\right)\approx h(\hat x_1^{\rm{r}})
\end{align*}
};

\node[node, below=10mm of S2] (S3) {%
\textbf{Averaged ROM} 
\begin{align*}
\dot{\bar x}_1^{\rm{r}} &= - \frac{\alpha_1 k_1}{2} \frac{d}{d x_1}\tilde J_{\text{L}}(\bar x_1^{\rm{r}}),
\quad \tilde J_{\text{L}}(\bar x_1^{\rm{r}})=J_{\text{L}}\big(\bar x_1^{\rm{r}},h(\bar x_1^{\rm{r}})\big)
\end{align*}
};

\draw[->, thick] (S0) -- (S1);
\draw[->, thick] (S1) -- (S2);
\draw[->, thick] (S2) -- (S3);

\node[lab, anchor=east] at ($(S0)!0.5!(S1)+(-6mm,0)$) {\large $\approx$};
\node[lab, anchor=east] at ($(S1)!0.5!(S2)+(-6mm,0)$) {\large $\approx$};
\node[lab, anchor=east] at ($(S2)!0.5!(S3)+(-6mm,0)$) {\large $\approx$};

\path (S0) -- (S1) node[midway, right=2mm, lab]{
\textit{Fast Lie--bracket averaging}\\approximated via $\omega_2$};
\path (S1) -- (S2) node[midway, right=2mm, lab] {\textit{Singular perturbation}\\
approximated via $\varepsilon={1}/{\alpha_2k_2}$};
\path (S2) -- (S3) node[midway, right=2mm, lab]{\textit{Slow Lie--bracket averaging}\\approximated via $\omega_1$};
\end{tikzpicture}

\caption{The system approximations used in the analysis, from the original nES dynamics (top) to the averaged ROM (bottom).}
\label{fig:state-xforms}
\end{figure}
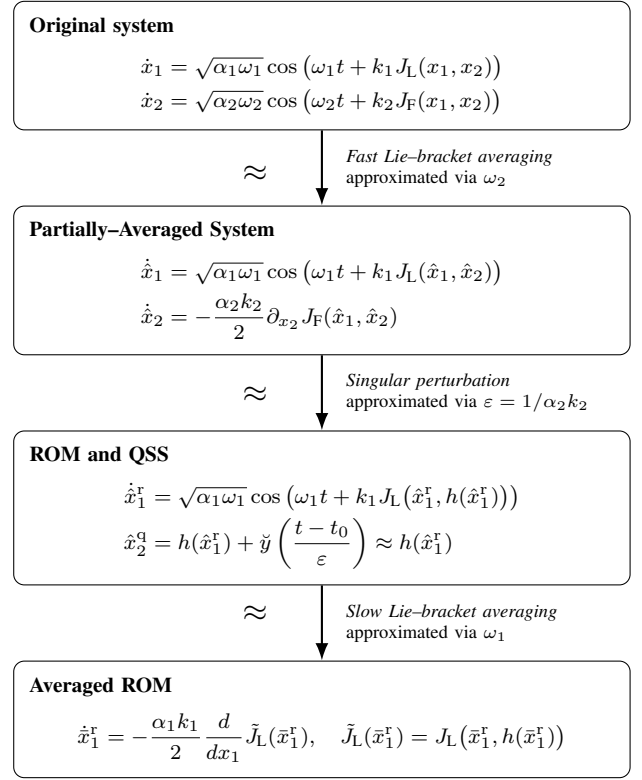

\begin{assumption}\label{asm:A1}
Let $J_{\text{L}}(x_1,x_2)\in C^2:\mathbb{R}\times\mathbb{R}\rightarrow \mathbb{R}$ and $J_{\text{F}}(x_1,x_2)\in C^3:\mathbb{R}\times\mathbb{R}\rightarrow \mathbb{R}$.
\end{assumption}

\begin{assumption}\label{asm:A2}
For each $x_1\in \mathbb{R}$, $\partial_{x_2}J_{\text{F}}(x_1,x_2)=0$ has a unique root $x_2=h(x_1)$. Moreover, $h\in C^1:\mathbb{R}\rightarrow\mathbb{R}$.
\end{assumption}

\begin{assumption}\label{asm:A3}
There exists $m_2>0$ such that
\begin{equation}
\partial_{x_2x_2}^2 J_{\text{F}}(x_1,x_2)\ge m_2,\qquad \forall (x_1,x_2)\in \mathbb{R}^2.
\end{equation}
\end{assumption}
For notational convenience, we define the reduced cost $\tilde{J}_1(x_1):=J_{\text{L}}(x_1,h(x_1))$, noting that
\begin{align}
    \frac{d}{dx_1}\tilde{J}_{\rm{L}}(x_1)=&\frac{d}{d x_1}J_{\text{L}}(x_1,h(x_1))=\frac{\partial}{\partial x_1}J_{\text{L}}(x_1,h(x_1))\notag    \\
    &+h'(x_1)\cdot\frac{\partial}{\partial x_2}J_{\text{L}}(x_1,h(x_1)),
\end{align}
where $h'(x_1)$ follows from the implicit function theorem applied to \(\partial_{x_2}J_{\text{F}}(x_1,h(x_1))=0\):
\begin{equation}\label{eq:hprime}
h'(x_1) = - \frac{\partial_{x_1 x_2}^2J_{\text{F}}\bigl(x_1,h(x_1)\bigr)}
{\partial_{x_2x_2}^2J_{\text{F}}\bigl(x_1,h(x_1)\bigr)}.
\end{equation}

\begin{assumption}\label{asm:A4}
There exists $x_1^*\in\mathbb{R}$ such that $\tilde J_{\text{L}}(x_1^*)=J_{\text{L}}(x_1^*,x_2^*)=0$, where we define $x_2^*:=h(x_1^*)$. Moreover, there exists a constant $ m_1>0$ such that
\begin{equation}
\frac{d^2}{d x_1^2}J_{\text{L}}(x_1,h(x_1))=\frac{d^2}{d x_1^2}\tilde{J}_1(x_1) \geq m_1, \quad \forall x_1\in \mathbb{R}.
\end{equation}
\end{assumption}

\section{System Approximations}\label{sec:system-approxs}
We now present a series of system approximations that will be leveraged in the stability analysis and later provide the practitioner with intuition for how the design parameters must be selected; Fig.~\ref{fig:state-xforms} provides a road map of how the approximations are used. We consider the following system approximations to hold on the finite time interval $t\in[t_0,t_0+T]$ for any $T>0$, and for an arbitrary compact set $\mathcal{K}\in\mathbb{R}^2$. 

\subsection{2D Lie-Bracket Averaging}\label{sec:avg2D-scaling}

Using the definitions found in \cite{durr2013lie}, we proceed to define the partially--averaged Lie--bracket system in the context of the nES dynamics. We can rewrite \eqref{eq:x1}--\eqref{eq:x2} in input--affine form \eqref{eq:gen_form_for_LBA}; let $x:=[x_1,x_2]^\top$ so that 
\begin{equation}\label{eq:avg2D:input-affine}
\dot x = b_0(t,x) + \sum_{i=1}^2 b_i(x)\sqrt{\omega_2} u_i(\omega_2 t),
\end{equation}
where $u_1(\theta)=\cos\theta$, $u_2(\theta)=\sin\theta$, and $\theta=\omega_2 t$ with
\begin{align}
b_0(t,x)=&
\begin{bmatrix}
\sqrt{\alpha_1\omega_1}\cos \big(\omega_1 t+k_1J_{\text{L}}(x)\big)\\
0
\end{bmatrix},\\
b_1(x)=&
\begin{bmatrix}
0\\
\sqrt{\alpha_2}\cos \big(k_2J_{\text{F}}(x)\big)
\end{bmatrix},\\
b_2(x)=&
\begin{bmatrix}
0\\
-\sqrt{\alpha_2}\sin \big(k_2J_{\text{F}}(x)\big)
\end{bmatrix}.
\end{align}
Let $\hat{x}=[\hat x_1,\hat x_2]^\top$. The partially--averaged Lie--bracket system associated with \eqref{eq:avg2D:input-affine} is then obtained by computing $\dot{\hat{x}} =b_0(t,\hat{x};\omega_1) + [b_1,b_2](\hat{x}) \nu_{21}$, which yields 
\begin{align}
\dot{\hat x}_1=&\sqrt{\alpha_1\omega_1}\cos \big(\omega_1 t+k_1J_{\text{L}}(\hat x_1,\hat x_2)\big)\label{eq:avg2D:LB-x1}\\
\dot{\hat x}_2=&-\frac{\alpha_2k_2}{2} \partial_{x_2}J_{\text{F}}(\hat x_1,\hat x_2)\label{eq:avg2D:LB-x2}
\end{align}
with initial condition $\left(\hat x_1(t_0), \hat x_2(t_0)\right)=\left(\xi_0,\eta_0\right)$. 

\begin{proposition}\label{prop:avg2D-scaling}
Consider \emph{fixed} $\alpha_1,k_1>0$, and any $\alpha_2,k_2>0$, $\omega_1>0$; by Theorem~\ref{thm:Dur-Avg}, there exists $\omega_2^*>0$ such that for $\omega_2>\omega_2^*$, the error between the original nES \eqref{eq:x1}--\eqref{eq:x2} and the partially--averaged Lie--bracket system \eqref{eq:avg2D:LB-x1}--\eqref{eq:avg2D:LB-x2} is of order 
\begin{equation}
\|x(t)-\hat{x}(t)\|=\mathcal O \Big(\frac{A(\alpha_2,k_2,\omega_1)}{\sqrt{\omega_2}}\exp \big(L_{\rm av}(t-t_0)\big)\Big),
\end{equation}
for $t\in[t_0,t_0+T]$, where $L_{\rm av}=\mathcal{O}(\sqrt{\omega_1} + \alpha_2 k_2)$, and $A(\alpha_2,k_2,\omega_1):=\max\{\alpha_2^{3/2}k_2^2,\alpha_2k_2\sqrt{\omega_1}\}$.
\end{proposition}

We do not include the proof of the proposition in order to conserve space; we note, however, that the result is derived from the proof of Theorem 1 found in \cite[Appendix B]{durr2013lie}, by simply computing the remainder terms $R_1$--$R_5$ from the cited text, then writing the error given by \cite[(B.9)]{durr2013lie} explicitly in terms of $(\alpha_2,k_2,\omega_1,\omega_2)$, where $k$ and $L$ from the cited text are given by $A(\alpha_2,k_2,\omega_1)$ and $L_{\rm{av}}$, respectively.

\subsection{Singular Perturbation}\label{sec:sp-scaling}

Using the definitions found in \cite[Theorem 11.1]{khalil_nonlinear_2002}, we introduce the necessary systems to define the ROM and QSS in the context of our nES system. 

\paragraph{Standard Singular Perturbation Form}
We are able to rewrite system \eqref{eq:avg2D:LB-x1}--\eqref{eq:avg2D:LB-x2} in standard singular perturbation form:
\begin{align}
\dot{\hat x}_1 = f(t,\hat x_1,\hat x_2):=& \sqrt{\alpha_1\omega_1} \cos(\omega_1 t + k_1 J_{\text{L}}(\hat x_1,\hat x_2)), \label{eq:x-sp-standard}\\
\varepsilon \dot{\hat x}_2 = g(\hat x_1,\hat x_2):=& -\tfrac12 \partial_{x_2}J_{\text{F}}(\hat x_1,\hat x_2),\label{eq:z-sp-standard}
\end{align}
with initial condition $(\hat x_1(t_0), \hat x_2(t_0))=(\xi_0,\eta_0)$, and where we recall that $\varepsilon = 1/(\alpha_2 k_2)$.

\paragraph{State Transformation}
Let $h(\hat x_1)$, referred to as the slow manifold, yield the roots of \eqref{eq:z-sp-standard} when $\varepsilon=0$ such that $\partial_{x_2}J_{\text{F}}(\hat x_1,h(\hat x_1))=0$ for all $\hat x_1$. We introduce the following state transformation $y=\hat x_2-h(\hat x_1)$ and let $y\in \mathcal{K}_y\subset \mathbb{R}$ be a compact set that contains the origin. In the new variables $(\hat x_1,y)$, 
\begin{align}
\dot{\hat x}_1 =& f(t,\hat x_1,y+h(\hat x_1)) \label{eq:state-transform-sys-xdot}\\
\varepsilon \dot y =& g(\hat x_1,y+h(\hat x_1))-\varepsilon h'(\hat x_1) f \big(t,\hat x_1, y+h(\hat x_1)\big) \label{eq:state-transform-sys-ydot}
\end{align}
with initial conditions $\hat x_1(t_0)=\xi_0$, and $y(t_0)=\eta_0-h(\xi_0)=:y_0$.

\paragraph{Boundary-Layer Model}
Let $\tau=(t-t_0)/\varepsilon$ and freeze the slow variable $\hat x_1$ by setting $\varepsilon=0$; the BLM is then given by
\begin{align} \label{eq:blm-sp-tf}
\frac{dy}{d\tau}=& g \big(\hat x_1, y+h(\hat x_1)\big)=  -\tfrac12 \partial_{x_2}J_{\text{F}}\big(\hat x_1, y+h(\hat x_1)\big).
\end{align}

\paragraph{The Reduced Order Model and Quasi--Steady State}
For $\varepsilon=0$, the ROM and QSS are given by 
\begin{align}
\dot{\hat x}_1^{\rm{r}}=&\sqrt{\alpha_1\omega_1}\cos\Big(\omega_1 t+k_1J_{\text{L}}\big(\hat x_1^{\rm{r}},h(\hat x_1^{\rm{r}})\big)\Big),\label{eq:x1-red}\\
\hat x_2^{\rm{q}}=&h(\hat x_1^{\rm{r}}(t))+\breve y \left(\frac{t-t_0}{\varepsilon}\right)\approx h(\hat x_1^{\rm{r}}(t)),\label{eq:x2-qss}
\end{align}
with initial condition $(\hat x_1^{\rm{r}}(t_0), \breve y(t_0))=(\xi_0, y_0)$, and where $\breve y(\tau)$ denotes the solution of the BLM \eqref{eq:blm-sp-tf}. We note that $\hat x_2^{\mathrm q}\approx h(\hat x_1^\mathrm{r}(t))$ because the BLM is exponentially stable to the origin, $y=0$, via Assumption~\ref{asm:A3} and application of Lemma~\ref{lem:scalar-exp-stable}.

\begin{proposition}\label{prop:sp-scaling}
Let Assumptions~\ref{asm:A1}--\ref{asm:A3} hold. Consider \eqref{eq:avg2D:LB-x1}--\eqref{eq:avg2D:LB-x2} written in standard singular perturbation form \eqref{eq:x-sp-standard}-\eqref{eq:z-sp-standard}, and the reduced problem \eqref{eq:x1-red}--\eqref{eq:x2-qss}. For \emph{fixed} $\alpha_1,k_1>0$, and for any $\omega_1>0$, there exists $\varepsilon^*>0$ such that for $0<\varepsilon<\varepsilon^*$ the slow state error is of order
\begin{align}
|\hat x_1(t)-\hat x_1^{\rm{r}}(t)|
=&\mathcal O \Big(\varepsilon e^{c\sqrt{\omega_1}}\Big),
\end{align}
and the fast state error, similarly, is of order
\begin{align}
&|\hat x_2(t)-\hat x_2^{\rm{q}}(t)|=\mathcal O \Big(\varepsilon e^{c\sqrt{\omega_1}}\Big),
\end{align}
for $t\in[t_0,t_0+T]$, where $c=L_{6,0}(t-t_0)$.
\end{proposition}

\begin{figure}[t!]
\centering
\includegraphics[width=0.9\linewidth]{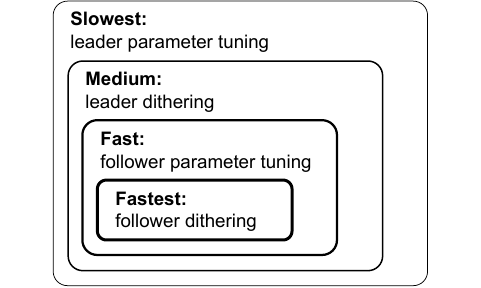}
\caption{Hierarchical nES time-scale diagram for the $n=2$ nested case.} 
\label{fig:timescale-diag}
\end{figure}

We do not include the proof of the proposition; we note, however, that the result is derived from the proof of Theorem 11.1 found in \cite{khalil_nonlinear_2002}. We satisfy the requirements of Theorem 11.1 via Assumptions~\ref{asm:A1}--\ref{asm:A3} and the use of Lemma~\ref{lem:scalar-exp-stable}, which certifies that the equilibrium $y=0$ of \eqref{eq:blm-sp-tf} is exponentially stable uniformly in $\hat x_1$ \cite[Definition 11.1]{khalil_nonlinear_2002}. We compute the error estimates given by \cite[(11.20)--(11.21)]{khalil_nonlinear_2002} explicitly in terms of $(\varepsilon,\omega_1)$, where in the cited text $L_6=\sqrt{\omega_1}L_{6,0}$.

\subsection{1D Lie--Bracket Averaging}\label{sec:avg1D-scaling}

We again work with the definitions found in \cite{durr2013lie} to define the averaged Lie--bracket reduced--order system in the context of our nES system. We rewrite the ROM as
\begin{align}\label{eq:avg:osc}
\dot{\hat x}_1^{\rm{r}}=&\sqrt{\omega_1}\Bigl[u_1(t,\theta)b_1(\hat x_1^{\rm{r}})+u_2(t,\theta)b_2(\hat x_1^{\rm{r}})\Bigr],
\end{align}
with initial condition $\hat x_1(t_0)=\xi_0$, where 
\begin{align}
b_1(\hat x_1^{\rm{r}}) &:= \sqrt{\alpha_1} \cos \bigl(k_1\tilde J_{\text{L}}(\hat x_1^{\rm{r}})\bigr), &
u_1(\theta) &:= \cos\theta, \label{eq:ROM-b1u1}    \\
b_2(\hat x_1^{\rm{r}}) &:= - \sqrt{\alpha_1} \sin \bigl(k_1\tilde J_{\text{L}}(\hat x_1^{\rm{r}})\bigr), &
u_2(\theta) &:= \sin\theta, 
\end{align}
where $\theta =\omega_1 t$, and recalling that $\tilde J_{\text{L}}(\hat x_1^{\rm{r}})=J_{\text{L}}(\hat x_1^{\rm{r}},h(\hat x_1^{\rm{r}}))$. The associated Lie--bracket averaged system is
\begin{equation}\label{eq:avg:x1-lb}
\dot{\bar x}_1^{\rm{r}}=[b_1,b_2](\bar x_1^{\rm{r}}) \nu_{21}= -\frac{\alpha_1k_1}{2} \frac{d}{dx_1}\tilde J_{\text{L}}(\bar x_1^{\rm{r}}),
\end{equation}
where $\bar x_1(t_0)=\xi_0$, and by the chain rule,
\begin{equation*}
\frac{d}{dx_1}\tilde J_{\text{L}}(\bar x_1^{\rm{r}})= \partial_{x_1}J_{\text{L}}\bigl(\bar x_1^{\rm{r}},h(\bar x_1^{\rm{r}})\bigr)+ \partial_{x_2}J_{\text{L}}\bigl(\bar x_1^{\rm{r}},h(\bar x_1^{\rm{r}})\bigr) h'(\bar x_1^{\rm{r}}),
\end{equation*}
where $(\cdot)':=\tfrac{d}{dx_1}$. The QSS is then evaluated using $\bar x_1^r$, yielding 
\begin{equation}
    \bar x_2^{\rm{q}}=h(\bar x_1^{\rm{r}}(t))+\breve y \left(\frac{t-t_0}{\varepsilon}\right)\approx h(\bar x_1^{\rm{r}}(t)).\label{eq:x2-bar-qss}
\end{equation}

\begin{proposition}\label{prop:avg1D-scaling}
Consider \emph{fixed} $\alpha_1,k_1>0$. By Theorem~\ref{thm:Dur-Avg}, there exists $\omega_1^*>0$ such that for $\omega_1>\omega_1^*$, the error between the ROM \eqref{eq:avg:osc} and the averaged Lie--bracket ROM \eqref{eq:avg:x1-lb} is of order 
\begin{equation}
|x_1(t)-\bar x_1(t)|
=\mathcal O \left(\frac{1}{\sqrt{\omega_1}}\right),
\end{equation}
for $t\in[t_0,t_0+T]$.
\end{proposition}

We do not include the proof of the proposition. We note, however, that the result is again derived by following the proof of Theorem 1 found in \cite[Appendix B]{durr2013lie}. Since the system is driftless, one simply computes the remainder terms $R_2$--$R_5$ from the cited text and then writes the error given by \cite[(B.9)]{durr2013lie} explicitly in terms of $\omega_1$.

\section{Main Result}
We first explain key aspects of the algorithm before presenting the main theorem; Fig.~\ref{fig:timescale-diag} shows the hierarchical structure of the \textbf{four timescales} present in the \(n=2\) case, which are most naturally understood from fastest to slowest: (1) The follower's dithering frequency must be large enough relative to the follower's parameter-tuning rate to induce partial-gradient follower dynamics. (2) On the fast adaptation timescale, the follower's parameter performs partial-gradient descent with respect to its own cost, \(\partial_{x_2}J_{\text{F}}(x_1,x_2)\), thereby producing the follower response to the leader's current decision. (3) The leader's dithering frequency must be chosen large enough so that, from the leader's perspective, the follower dynamics are effectively algebraic and the follower's response is represented by the static map \(x_2=h(x_1)\). (4) On the slowest timescale, the leader updates its parameter using the follower's induced response, and the dynamics reflect descent along \(\frac{d}{dx_1}J_{\text{L}}(x_1,h(x_1))\). Therefore, the nES design parameters must be chosen to preserve the four-timescale hierarchy so as to induce the Stackelberg leader--follower behavior of the algorithm. 

We now present the main theoretical result of the paper for the $n=2$ nested case of the nES dynamics. 
\begin{theorem}\label{thm:Avg-SP-Stack}
Under Assumptions~\ref{asm:A1}--\ref{asm:A4}, consider the nES system \eqref{eq:x1}--\eqref{eq:x2} with $k_i,\alpha_i>0$ and $\omega_i>0$ for $i\in\{1,2\}$. There exists $\upsilon^*>0$ such that for every $0<\upsilon<\upsilon^*$ there exist thresholds $\omega_1^*(\upsilon)$, $\varepsilon^*(\upsilon,\omega_1)$, and $\omega_2^*(\upsilon,\omega_1,\varepsilon)$ so that for $\omega_1>\omega_1^*$, $0<\varepsilon<\varepsilon^*$, and $\omega_2>\omega_2^*$, \eqref{eq:x1}--\eqref{eq:x2} is $\upsilon$-SPUAS to the point $(x_1^*,x_2^*)=(x_1^*,h(x_1^*))$.
\end{theorem}

\begin{proof} 
We break the proof into the following steps: \textbf{Step 1}: Using the system approximations from Section~\ref{sec:system-approxs}, we construct explicit scaling laws for the design parameters $(\alpha_1k_1,\omega_1,\alpha_2k_2,\omega_2)$ in terms of a small parameter $\upsilon>0$. By combining error bounds from Propositions~\ref{prop:avg2D-scaling}--\ref{prop:avg1D-scaling}, we show that all approximation errors vanish as $\upsilon\to0$, which establishes the CTP between the original dynamics and the averaged ROM--QSS pair. \textbf{Step 2}: Under the stated convexity and regularity assumptions, we show that the averaged ROM and QSS subsystem is exponentially stable. \textbf{Step 3}: We combine the CTP result from Step~1 with the exponential stability established in Step~2. Applying Theorem~\ref{thm:MA-bridge}, we conclude that the nES dynamics are $\upsilon$-SPUAS to the point $(x_1^*,x_2^*)$.

\textbf{Step 1:} 
Applying the system approximations presented in Section~\ref{sec:system-approxs}, we approximate the nES \eqref{eq:x1}--\eqref{eq:x2} dynamics $(x_1,x_2)$ with the averaged ROM \eqref{eq:avg:x1-lb} and the QSS \eqref{eq:x2-qss} pair $(\bar x_1^{\rm{r}}, \bar x_2^{\rm{q}})$, uniformly for $t\in[t_0,t_0+T]$. 
We now devise a scaling law for the design parameters $(\omega_1,\varepsilon,\omega_2)$ using Propositions~\ref{prop:avg2D-scaling}--\ref{prop:avg1D-scaling}. 
Let $\upsilon>0$ be a small parameter.

\textbf{1D Averaging Error:}
From Proposition~\ref{prop:avg1D-scaling},
\begin{equation}
|\hat x_1^{\rm{r}}(t)-\bar x_1^{\rm{r}}(t)|
=\mathcal O \left(\frac{1}{\sqrt{\omega_1}}\right).
\end{equation}
To ensure that this error vanishes as $\upsilon\to0$, we impose
\begin{equation}
\omega_1 > \omega_1^*:= \frac{C_1}{\upsilon^2} 
\end{equation}
for some constant $C_1>0$ independent of $\upsilon$ where $\omega_1^*$ grows as $\upsilon \to 0$.

\begin{figure}[t!]
\centering
\includegraphics[width=\columnwidth]{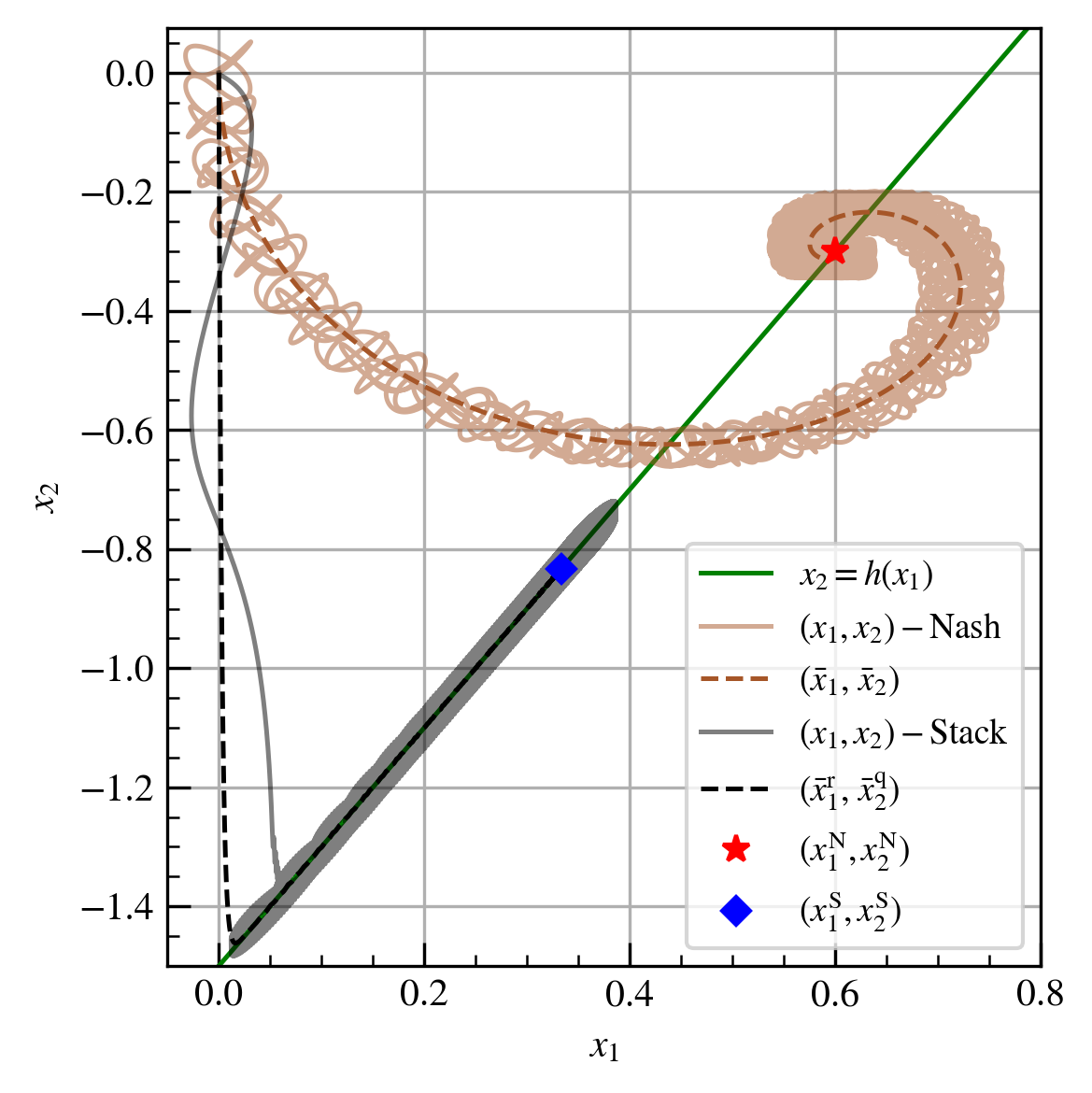}
\caption{Phase space plot for the example discussed in Section \ref{sec:Simple-Example}, comparing the practical convergence to the Nash (brown traces) versus Stackelberg (grey traces) equilibrium under the nES dynamics.} 
\label{fig:nES-example-phase}
\end{figure}

\textbf{Singular Perturbation Error:}
From Proposition~\ref{prop:sp-scaling},
\begin{equation}
|\hat x_1(t)-\hat x_1^{\rm{r}}(t)|
=
\mathcal O \left(\varepsilon e^{c\sqrt{\omega_1}}\right).
\end{equation}
Since $\sqrt{\omega_1}=\mathcal O(1/\upsilon)$ under the scaling law for $\omega_1$, we rewrite the error in terms of $\upsilon$:
\begin{equation}
|\hat x_1(t)-\hat x_1^{\rm{r}}(t)|
=
\mathcal O \left(\varepsilon e^{c/\upsilon}\right).
\end{equation}
Hence we require
\begin{equation}
\varepsilon < \varepsilon^*:= C_2 e^{-2c/\upsilon}
\end{equation}
for some constant $C_2>0$, where $\varepsilon^*$ shrinks as $\upsilon \to 0$. 
Since $\varepsilon={1}/{(\alpha_2 k_2)}$, and for simplicity we consider \emph{fixed} $\alpha_2>0$, implying 
\begin{equation}
k_2 > k_2^* := \frac{1}{C_2\alpha_2} e^{2c/\upsilon},
\end{equation}
where $k_2$ grows as $\upsilon\to 0$.

\textbf{2D Averaging Error:}
From Proposition~\ref{prop:avg2D-scaling},
\begin{equation}
\|x(t)-\hat{x}(t)\|=\mathcal O \Big(\frac{A(\alpha_2,k_2,\omega_1)}{\sqrt{\omega_2}}\exp \big(L_{\rm av}(t-t_0)\big)\Big).
\end{equation}
Using the above bounds on $\omega_1$ and $k_2$, we obtain
\begin{equation}
A(\alpha_2,k_2,\omega_1)=\mathcal O \left(e^{4c/\upsilon}\right), \quad L_{\rm av}=\mathcal O \left(e^{2c/\upsilon}\right).
\end{equation}
Hence, for $t\in[t_0,t_0+T]$,
\begin{equation}
\|x(t)-\hat x(t)\|
=
\mathcal O \left(
\frac{1}{\sqrt{\omega_2}}
\exp \big(e^{2c/\upsilon}\big)
\right).
\end{equation}
Therefore require
\begin{equation}
\omega_2 >
\omega_2^* :=
C_3 \exp \big(2e^{2c/\upsilon}\big),
\end{equation}
for some constant $C_3>0$.

 \begin{figure}[t!]
\centering
\includegraphics[width=\columnwidth]{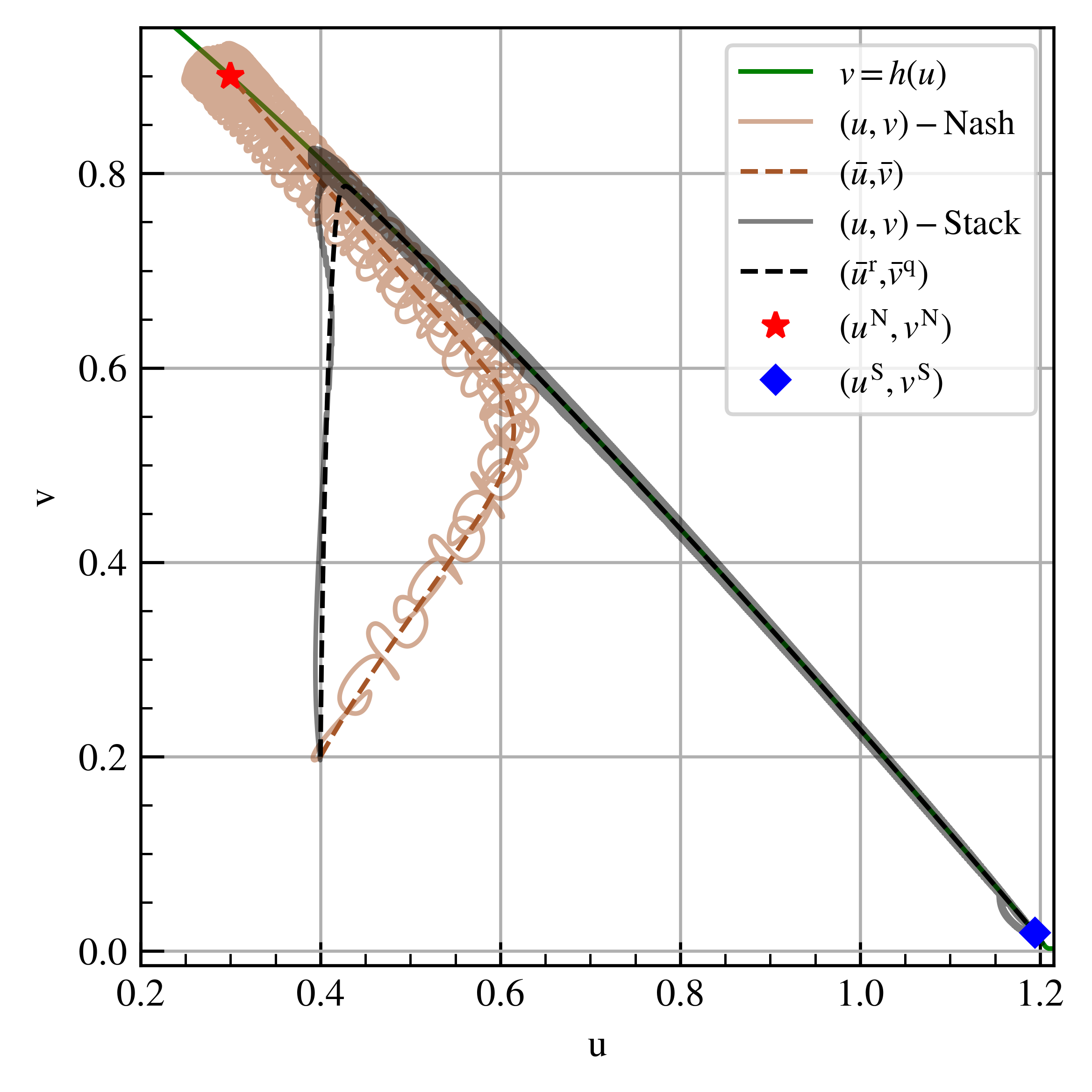}
\caption{Phase space plot for the Fish War example discussed in Section \ref{sec:fishwar}, comparing the practical convergence to Nash (brown traces) versus Stackelberg (grey traces) equilibrium under the nES dynamics.} 
\label{fig:nES-fish-war-phase}
\end{figure}

In summary, for \emph{fixed} $\alpha_1,\alpha_2>0$, $k_1>0$, and under the parameter scaling laws $\omega_1 > \omega_1^*$, $k_2 > k_2^*$, and $\omega_2 > \omega_2^*$,
where 
\begin{align}
    \omega_1^*= \frac{C_1}{\upsilon^2},\quad k_2^*=\frac{e^{2c/\upsilon}}{C_2\alpha_2},\quad \omega_2^* = C_3 \exp \big(2 e^{2c/\upsilon}\big),
\end{align}
all of the approximation errors vanish as $\upsilon\to0$. Therefore, for any $\delta>0$ there exists $\upsilon^*>0$ such that for $0<\upsilon<\upsilon^*$, by the triangle inequality, 
\begin{equation}
    |x_1(t)-\bar x_1^{\rm{r}}(t)|<\delta,\qquad|x_2(t)-\hat x_2^{\rm{q}}(t)|<\delta,
\end{equation}
which satisfies the definition of the CTP given by Definition~\ref{def:CTP}.

\textbf{Step 2:} 
By Assumption~\ref{asm:A4}, $\tilde J_{\text{L}}(x_1)$ is $ m_1$-strongly convex for all $x_1$ and hence by Lemma~\ref{lem:scalar-exp-stable} the averaged ROM \eqref{eq:avg:x1-lb} is exponentially stable. 

\textbf{Step 3}: By combining the CTP result from Step 1 with the exponential stability of the ROM from Step 2, Theorem \ref{thm:MA-bridge} implies that the point $(x_1^*,x_2^*)$ of the original nES dynamics is $\upsilon$-SPUAS.
\end{proof}

By Theorem~\ref{thm:Avg-SP-Stack}, the nES dynamics are $\upsilon$-SPUAS to the point 
\(
(x_1^*,x_2^*)=(x_1^*,h(x_1^*)).
\)
Under the assumptions, this point is in fact a strict SE with $x_1$ as the leader and $x_2$ as the follower. Indeed, Assumption~\ref{asm:A3} guarantees that for each $x_1$ the follower’s problem 
\(
\min_{x_2} J_{\text{F}}(x_1,x_2)
\)
is strongly convex in $x_2$, so that the best-response map $h(x_1)$ is single-valued and locally minimizing. Assumption~\ref{asm:A4} ensures that the reduced cost 
\(
\tilde J_{\text{L}}(x_1)=J_{\text{L}}(x_1,h(x_1))
\)
is strongly convex, so that $x_1^*$ is the unique minimizer. Consequently, $(x_1^*,x_2^*)$ satisfies the SE conditions introduced in \eqref{eq:x^S condition}. 

\begin{figure*}[t!]
\centering
\includegraphics[width=\textwidth]{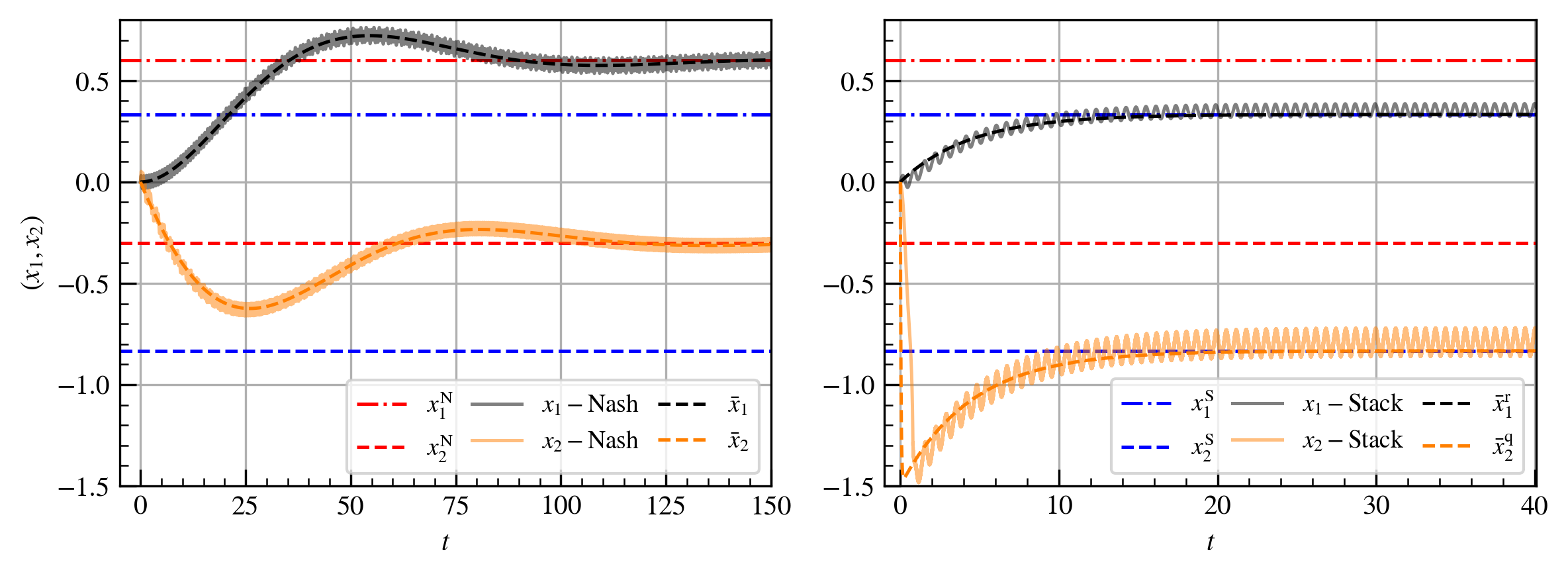}
\caption{Time series plots for the example in Section \ref{sec:Simple-Example}, comparing the practical convergence to the Nash (left) versus Stackelberg (right) equilibria.}
\label{fig:nES-example-time}
\end{figure*}

In a previous study, which also applied nES to a two-player game \cite{Ratto2026ACC}, the same dynamics \eqref{eq:x1}--\eqref{eq:x2} were shown to converge to a neighborhood of the NE. The distinction lies not in the closed-loop dynamics themselves, but in the imposed scaling laws on the design parameters; which as a result requires a different permutation of the invoked averaging and singular perturbation theorems to analyze the stability and convergence of the system. 
In the Nash analysis, the design parameter scaling yielded simultaneous coupled partial--gradient dynamics for both players given by 
\begin{align} 
    \dot{\bar x}_1=-\frac{\alpha_1k_1}{2}\partial_{x_1}J_1(\bar x_1, \bar x_2),  \label{eq:Nash-dynamics-x1}\\
    \varepsilon\dot{\bar x}_2=-\frac{\alpha_2k_2}{2}\partial_{x_2}J_2(\bar x_1, \bar x_2), \label{eq:Nash-dynamics-x2}
\end{align}
for which the stationary NE conditions $\partial_{x_1}J_1(x_1^{\textrm N}, x_2^{\textrm N})=0$, and $\partial_{x_2}J_2(x_1^{\textrm N}, x_2^{\textrm N})=0$ are satisfied, where $(x_1^{\textrm N}, x_2^{\textrm N})$ denotes the unique NE \cite{facchinei2003finite,wu2011games}.

In contrast, the present analysis imposes a strict time-scale hierarchy on the design variables. The two-dimensional Lie--bracket averaging is performed in a manner that produces a partially averaged system amenable to singular perturbation. The scaling is then chosen so that strong time-scale separation forces the follower dynamics to converge rapidly to the best-response manifold $x_2=h(x_1)$, after which a one-dimensional Lie--bracket averaging is used, resulting in the leader dynamics evolving along the full gradient of the reduced cost $\frac{d}{d x_1}J_{\text{L}}(x_1,h(x_1))$. The sequenced analysis using averaging and singular perturbation theorems, as presented in the proof of Theorem~\ref{thm:Avg-SP-Stack} and shown in Fig.~\ref{fig:state-xforms}, reflects the hierarchical structure characteristic of a Stackelberg game, leading to the SE condition $\frac{d}{d x_1}J_{\text{L}}(x_1^{\textrm S},h(x_1^{\textrm S}))=0$ and $x_2^{\textrm S}=h(x_1^{\textrm S})$. Thus, the same nES dynamics can result in convergence to either a neighborhood of a NE or a SE, depending on the scaling of the design parameters.

\section{A Simple Example}\label{sec:Simple-Example}

We proceed to use the nES algorithm \eqref{eq:x1}--\eqref{eq:x2} to demonstrate the algorithm's ability to converge to either a neighborhood of the unique NE or SE depending on the relative scaling of the design parameters. Consider the quadratic costs
\begin{align}
    J_{\text{L}}(x_1,x_2) =& \tfrac{1}{2}x_1^2 + 2x_1x_2, \label{eq:ex-J1}     \\
    J_{\text{F}}(x_1,x_2) =& \tfrac{1}{2}(x_2 - 2x_1 + 1.5)^2, \label{eq:ex-J2}
\end{align}
and it can be verified that \(J_{\text{L}}\) and \(J_{\text{F}}\) satisfy Assumptions~\ref{asm:A1}--\ref{asm:A4}. The corresponding Nash and Stackelberg equilibria are
$(x_1^{\textrm N},x_2^{\textrm N})=(0.6,-0.3)$ and $(x_1^{\textrm S},x_2^{\textrm S})=(0.3\bar3,-0.8\bar3)$.

To illustrate practical convergence to the NE, we set $\varepsilon=0.75$, \(\alpha_1=\alpha_2=1.0\times10^{-2}\), and \(k_1=k_2=5\). Choosing \(\omega_1=10.0\) rad/s and \(\omega_2=\sqrt{2}\,\omega_1\) induces the partial-gradient dynamics in \eqref{eq:Nash-dynamics-x1}--\eqref{eq:Nash-dynamics-x2}. As shown in earlier work, the resulting nES dynamics are SPUAS with respect to the NE \cite{Ratto2026ACC}, where the trajectories are denoted using $(x_1, x_2)$-Nash in Fig.~\ref{fig:nES-example-phase} and Fig.~\ref{fig:nES-example-time}.

For the Stackelberg case, we retain \(\alpha_1=1.0\times10^{-2}\), \(k_1=2\), and \(\omega_1=10.0\) rad/s. The follower's adaptation parameters are set to $\alpha_2=0.1$, \(k_2=500\), and the follower's dithering frequency is set to $\omega_2=50\cdot\sqrt{2}\,\omega_1$ rad/s. By Theorem~\ref{thm:Avg-SP-Stack}, the original nES dynamics are approximated by the averaged reduced-order model \eqref{eq:avg:x1-lb} and are \(\upsilon\)-SPUAS with respect to the SE, where the trajectories are denoted using $(x_1, x_2)$-Stack in Fig.~\ref{fig:nES-example-phase} and Fig.~\ref{fig:nES-example-time}.

Notice that the design parameters satisfy the hierarchical structure shown in Fig.~\ref{fig:timescale-diag},
\begin{equation}\label{eq:design-rule-of-thumb}
    \alpha_1k_1 \ll \omega_1 \ll \alpha_2k_2 \ll \omega_2,
\end{equation}
which naturally fits the leader-follower structure of Stackelberg games. The follower operates on a timescale determined by selecting \(\omega_2\) and adaptation gain \(\alpha_2k_2\), resulting in partial-gradient descent dynamics with respect to the follower's cost $\partial_{x_2}J_{\text{F}}(x_1,x_2)$. The leader, in turn, chooses \(\omega_1\) and \(\alpha_1k_1\), inducing a time-scale separation. The leader then evolves on a slower timescale, allowing it to respond with information from the follower's choice and thereby induce descent along $\frac{d}{dx_1}J_{\text{L}}(x_1,h(x_1))$. 

\begin{figure*}[t!]
\centering
\includegraphics[width=\textwidth]{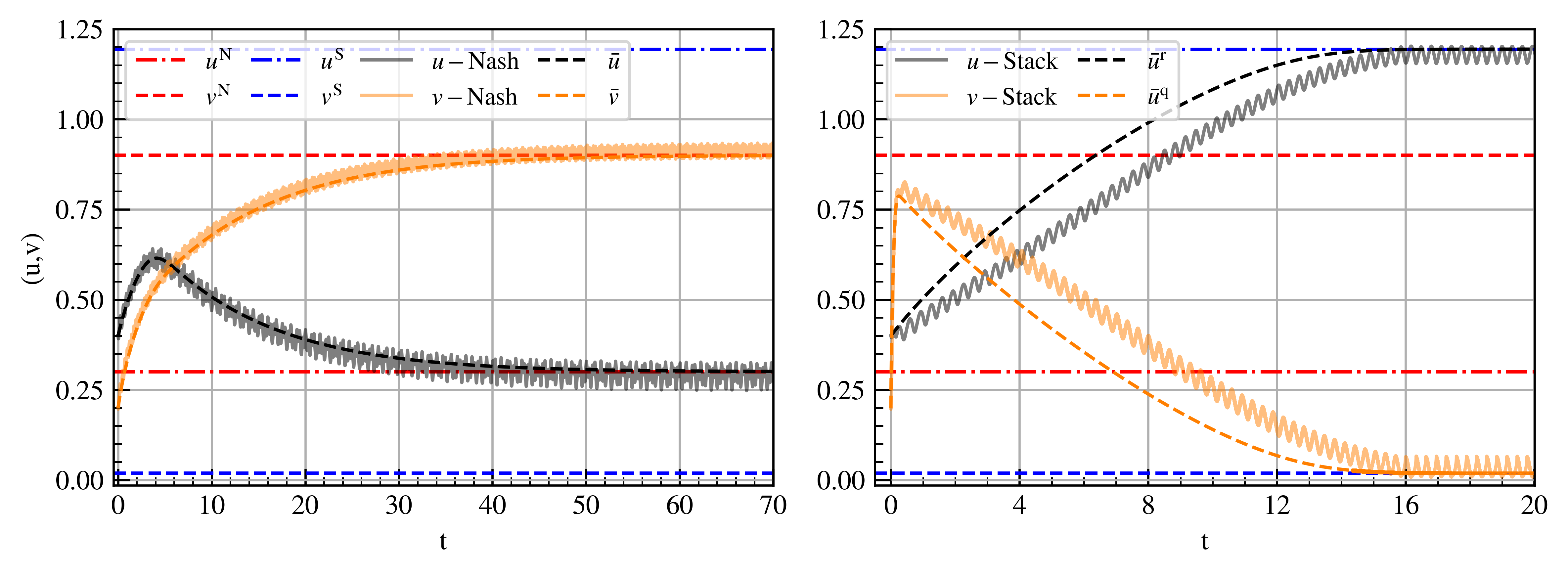}
\caption{Time series plots for the Fish War example in Section \ref{sec:fishwar}, showing the practical convergence to the Nash (left) versus Stackelberg (right) equilibria.} 
\label{fig:nES-fish-war-time}
\end{figure*}

\section{Fish War Example}
\label{sec:fishwar}

We now consider a problem from the game--theory literature known as the \emph{Fish War}, reported in \cite{vallee1999off}, originally studied in \cite{levhari1980great} and \cite{li1987distributed}. We adopt the same notation for the player variables as in \cite{vallee1999off}, where $(u,v):=(x_1,x_2)$ denote the current consumption levels of two countries, which incur costs
\begin{align}
J_{\rm{L}}(u,v) &= -\log u \;-\; \beta_L \log\!\Big(x - u - v^{\mu_L}\Big)^{\tau}, \label{eq:fishJL}\\
J_{\rm{F}}(u,v) &= -\log v \;-\; \beta_F \log\!\Big(x - v - u^{\mu_F}\Big)^{\tau}, \label{eq:fishJF}
\end{align}
over the feasible set
\begin{equation*}
\mathcal{D}=\Big\{(u,v): u\ge 0,\; v\ge 0,\; u+v^{\mu_L}\le x,\; u^{\mu_F}+v\le x\Big\}.
\end{equation*}
Here $x>0$ represents the fish population, $\beta_L,\beta_F\in(0,1]$ are discount factors, $\tau\in(0,1)$, and $\mu_L,\mu_F>1$ are curvature exponents \cite{vallee1999off}. We use the same parameter values as in \cite{vallee1999off},
\[
(\tau,\mu_L,\mu_F,\beta_L,\beta_F,x)=(0.2852,1.1,1.2,0.8,0.48,1.259),
\]
and the reported Nash and Stackelberg equilibria are $(u^{\rm{N}},v^{\rm{N}}) = (0.3,\,0.9)$ and $(u^{\rm{S}},v^{\rm{S}}) = (1.19426,\,0.01896)$, with corresponding costs reported in \cite{vallee1999off}.

Despite not meeting Assumptions~\ref{asm:A1}--\ref{asm:A4}, the mechanism behind our main result applies \emph{locally} because fundamentally Theorem~\ref{thm:Avg-SP-Stack} depends on: (i) local regularity and bounded derivatives on the compact set explored by trajectories, (ii) a locally unique follower best response $v=h(u)$ with local exponential stability of the follower’s averaged dynamics, and (iii) local stability of the leader’s reduced averaged dynamics. 

To induce practical convergence to the NE, we choose $\varepsilon=0.75$, $\alpha_1=\alpha_2=1.0\times 10^{-2}$, $k_1=k_2=10$, with dithering frequencies $\omega_1=20$ rad/s, and $\omega_2=\sqrt{2}\,\omega_1$ rad/s. As in the quadratic example, this choice of frequencies yields the partial-gradient structure associated with the Nash-seeking dynamics. Since the two players evolve on comparable timescales, neither player fully settles to a local best response of the other before the other updates; the resulting behavior is consistent with simultaneous adaptation toward the NE. The corresponding phase space and time series plots in Fig.~\ref{fig:nES-fish-war-phase} and Fig.~\ref{fig:nES-fish-war-time} show convergence to a neighborhood of the reported NE $(u^{\rm N},v^{\rm N})$, where the trajectories are denoted using $(u, v)$-Nash.

To demonstrate practical convergence to the SE, we retain the leader parameters $\alpha_1=1.0\times 10^{-2}$, $k_1=10$, $\omega_1=20$ rad/s, but increase the follower adaptation parameters and dithering frequency to $\alpha_2=5.0\times 10^{-2}$, $k_2=100$, and $\omega_2=50\cdot\sqrt{2}\,\omega_1$. These parameters again enforce the hierarchical separation between the leader and follower dynamics summarized by \eqref{eq:design-rule-of-thumb}. The resulting trajectories, shown in Fig.~\ref{fig:nES-fish-war-phase} and Fig.~\ref{fig:nES-fish-war-time}, converge to a neighborhood of the reported SE $(u^{\rm S},v^{\rm S})$, where the trajectories are denoted using $(u, v)$-Stack. Although the Fish War does not satisfy the standing assumptions globally, the simulations demonstrate that the same nES architecture can still recover either Nash-like or Stackelberg-like behavior through appropriate parameter scaling.

\section{Conclusion}
The nES algorithm was shown to practically converge to a SE, as compared to a previous study showing convergence to a neighborhood of a NE. The results reveal an important conceptual insight: convergence to an equilibrium type, Nash versus Stackelberg, is not determined by modifying the feedback structure, but by enforcing a hierarchical structure through parameter scaling. The nES algorithm provides a unifying model--free framework capable of implementing either simultaneous or hierarchical optimization in multi--agent systems. Future work will focus on analyzing the general form of nES for $n$-nested games, in addition to studying the application of nES to time-varying dynamical systems with $k$-dimensional states for particle accelerators. 

\bibliographystyle{ieeetr}
\bibliography{ref}
\end{document}